\documentclass[11pt]{article}

\usepackage[utf8]{inputenc} 
\usepackage{url}

\usepackage{amsmath,amsfonts,amssymb}
\numberwithin{equation}{section}
\usepackage{amsthm}
\usepackage{graphicx}
\usepackage{enumerate}
\usepackage{multirow,bigdelim}
\usepackage{mathtools}
\usepackage{soul} 
\usepackage{pstricks}
\usepackage[ruled, boxed,  linesnumbered, norelsize]{algorithm2e}
\usepackage[margin=3cm]{geometry}
\usepackage[colorinlistoftodos,prependcaption,textsize=footnotesize]{todonotes}
\usepackage{array}
\usepackage[normalem]{ulem}
\usepackage{url}
\usepackage{hyperref}
\usepackage{enumerate}
\usepackage{mathtools}
\hypersetup{
	colorlinks,
	linkcolor={red!50!black},
	citecolor={blue!50!black},
	urlcolor={blue!80!black}
}

\newcommand{\norm}[1]{\lVert{#1}\rVert}

\newcommand{\stdCone}{ {\mathcal{K}}}

\newcommand{\stdFace}{ \mathcal{F}}

\renewcommand{\Re}{\mathbb{R}}

\newcommand{\face}{\mathrel{\unlhd}}

\DeclareMathOperator{\lspan}{span}

\DeclareMathOperator{\dist}{dist}

\DeclareMathOperator{\interior}{int}
\DeclareMathOperator{\reInt}{ri}

\newcommand{\RR}{\mathbb{R}}
\newcommand{\R}{\mathbb{R}}

\newtheorem{definition}{Definition}[section]
\newtheorem{lemma}[definition]{Lemma}
\newtheorem{proposition}[definition]{Proposition}

\newtheorem{corollary}[definition]{Corollary}

\newtheorem{theorem}[definition]{Theorem}
\newtheorem*{proposition*}{Proposition}
\theoremstyle{remark}

\title{Hyperbolicity cones are amenable}
\author{Bruno F.\ Louren\c{c}o\thanks{Department of Statistical Inference and Mathematics, Institute of Statistical Mathematics, Japan. Email:~\texttt{bruno@ism.ac.jp}} \and Vera Roshchina\thanks{School of Mathematics and Statistics, UNSW Sydney, Australia. Email:~\texttt{v.roshchina@unsw.edu.au}} \and James Saunderson\thanks{Department of Electrical and Computer Systems Engineering, Monash University, Australia. Email:~\texttt{james.saunderson@monash.edu}}}

\begin{document}
\maketitle

\begin{abstract}
Amenability is a notion of facial exposedness for convex cones that is stronger
	than being facially dual complete (or `nice') which is, in turn,
	stronger than merely being facially exposed. Hyperbolicity cones are a
	family of algebraically structured closed convex cones that contain all
	spectrahedral cones (linear sections of positive semidefinite cones) as
	special cases. It is known that all spectrahedral cones are amenable.  We
	establish that all hyperbolicity cones are amenable. As part of the
	argument, we show that any face of a hyperbolicity cone is a
	hyperbolicity cone. As a corollary, we show that the intersection of
	two hyperbolicity cones, not necessarily sharing a common relative
	interior point, is a hyperbolicity cone.
\end{abstract}


\section{Introduction}
\label{sec:intro}
A homogeneous polynomial $p$ of degree $d$ in $n$ variables 
is \emph{hyperbolic with respect to $e\in \Re^n$} if $p(e)>0$ and, for each $x\in \Re^n$
the univariate polynomial $t\mapsto p(te-x)$ has only real zeros. Associated with a hyperbolic polynomial
is the (closed) hyperbolicity cone
\begin{equation}
	\label{eq:hypcone}
	\Lambda_+(p,e) \coloneqq \{x\in \Re^n \mid \textup{all zeros of $t\mapsto p(te-x)$ are nonnegative}\}.
\end{equation}
Hyperbolicity cones are convex cones that generalize \emph{spectrahedral cones}. Spectrahedral cones
are linearly isomorphic to the intersection of a cone of
symmetric positive semidefinite matrices with a linear space. 
Indeed, if $A_1,A_2,\ldots,A_n$ are symmetric matrices and 
$e\in \RR^n$ is such that $\sum_{i=1}^{n}A_ie_i$ is positive definite, then 
$p(x) = \det(\sum_{i=1}^{n}A_ix_i)$ is hyperbolic with respect to $e$, and the
associated hyperbolicity cone is the spectrahedral cone 
$\{x\in \RR^n\;:\;\sum_{i=1}^{n}A_ix_i \;\;\textup{is positive semidefinite}\}$.

The \emph{generalized Lax conjecture} asserts that every hyperbolicity cone is a 
spectrahedral cone. This is known to hold (in a stronger form) in dimension
three~\cite{helton2007linear,lewis2005lax}, but is unresolved in general. 
Examples of hyperbolic polynomials with hyperbolicity cones that are not currently 
known to be spectrahedral include
mixed discriminants of partially specified tuples of 
symmetric matrices (see, e.g.,~\cite[Section 6]{saunderson2020terracini}), 
certain hyperbolic polynomials associated 
with (hyper)graphs~\cite{amini2018non}, and certain cubic hyperbolic polynomials 
associated with the size of the largest clique in a graph~\cite{saunderson2019certifying}.

If the generalized Lax conjecture
were true, then any geometric property of spectrahedral cones should also hold for 
hyperbolicity cones. In this paper, we consider properties related to the faces (see Section~\ref{sec:prelim-error} for a definition) 
of hyperbolicity cones that are known to hold for spectrahedral cones. 

A closed convex cone $\stdCone$ is \emph{amenable} if, for each face $\stdFace$ of $\stdCone$, 
there exists a constant $\kappa>0$ such that 
\begin{equation}\label{eq:amenable}
 \textup{dist}(x,\stdFace) \leq \kappa\, \textup{dist}(x,\stdCone) \quad
\textup{for all $x\in \textup{span}(\stdFace)$}.
\end{equation}
The notion of amenable cones was introduced in~\cite{L19} in the context of
establishing error bounds for conic feasibility problems in the absence of
constraint qualifications. In our context, we are interested in amenability
because it is a strong notion of facial exposedness for convex cones. 

Every amenable cone is nice (or facially dual complete)~\cite[Proposition~13]{L19}, and every nice cone
is facially exposed~\cite[Theorem 3]{pataki}. All three notions coincide in dimension 
three (due to~\cite[Theorem 3.2]{PL88}).
However, in dimension four there are facially exposed cones that are not
nice~\cite{Vera}, and nice cones that are not amenable~\cite[Theorem 5.3]{LRS20}. Therefore amenability 
is a considerable strengthening of the property of being facially exposed.

It was recently established that spectrahedral cones are amenable~\cite[Corollary 3.5]{LRS20}. 
In light of the Lax conjecture, it is natural
to ask whether hyperbolicity cones are also amenable. Our main result is the following.
\begin{theorem}
	\label{thm:hyp_am}
	Every hyperbolicity cone is amenable.
\end{theorem}
Prior to this, it was known that hyperbolicity cones are facially exposed~\cite[Theorem~23]{Re06}, but it
was not known whether every hyperbolicity cone is nice, let alone amenable. A simple 
consequence of Theorem~\ref{thm:hyp_am} and~\cite[Proposition~13]{L19} is that every hyperbolicity cone is nice. 

Another basic geometric fact about the faces of spectrahedral cones is that they are, themselves, spectrahedral cones. This follows from the fact that the faces of the positive semidefinite cone are positive semidefinite cones of smaller dimension. En route to the proof of Theorem~\ref{thm:hyp_am}, we show that faces of hyperbolicity cones are, 
themselves, hyperbolicity cones (Corollary~\ref{cor:faces}). Based on this result, we also show that the
intersection of two hyperbolicity cones is, again, a hyperbolicity cone (Corollary~\ref{cor:intersection}). This is
a well-known fact under the additional assumption that the hyperbolicity cones share a relative interior point 
(and hence a direction of hyperbolicity). However, we could not find this result in the literature in 
the more subtle case where the relative interiors of the cones do not intersect.

\paragraph{Outline} 
In Section~\ref{sec:prelim-error}, we recall the
required definitions from convex geometry, and state a standard error bound
result that will be needed for the proof of Theorem~\ref{thm:hyp_am}. 
In Section~\ref{sec:prelim-hyp} we briefly summarize certain key facts and definitions
about hyperbolic polynomials, hyperbolicity cones, and their derivative
relaxations, that will be used throughout.  
In Section~\ref{sec:hyp-faces}, we show that faces of hyperbolicity cones are, themselves, 
hyperbolicity cones, and that the intersection of two hyperbolicity cones is always a 
hyperbolicity cone. We establish these facts via a description of a face of a 
hyperbolicity cone as the (proper) intersection of the span of the face with an appropriate derivative relaxation
of the hyperbolicity cone. We make this description completely explicit, in Section~\ref{sec:sp-f}, by giving
a concrete description of the span of a face of a hyperbolicity cone given a point in its relative interior.
In Section~\ref{sec:pf-hyp_am} we use our representation of faces, together
with the error bound stated in Section~\ref{sec:prelim-error}, to show that every hyperbolicity cone
is amenable. In Section~\ref{sec:discussion} we discuss some further questions related to the 
faces of hyperbolicity cones.

\section{Preliminaries}
\label{sec:prelim}

\subsection{Convex analysis and error bounds}
\label{sec:prelim-error}
Let $C \subseteq \RR^n$ be a closed convex set. In 
what follows we denote the interior, relative interior and the span of $C$ by $\interior C$, $\reInt C$ and $\lspan C$, respectively. We also recall that a 
\emph{face} of $C$ is a closed convex set $\stdFace$ contained in $C$ satisfying the following property: whenever $x,y \in C$ are such that 
$ \alpha x + (1-\alpha) y \in \stdFace $ for some $\alpha \in (0,1)$, we have $x,y \in \stdFace$.
In this case, we write $\stdFace \face C$.

In what follows, we assume that $\RR^n$ is equipped with some norm $\norm{\cdot}$ and denote by $\dist(x,C)$ the distance from 
$x$ to $C$ so that $\dist(x,C) \coloneqq \inf \{\norm{x-y} \mid y \in C \}.$
Finally, we  need the following error bound result.

\begin{proposition}[Linear regularity under a constraint qualification]\label{prop:error} Let $L$ be a linear subspace and $\stdCone$ a closed convex cone with $L\cap \interior \stdCone\neq \emptyset$. Then there exists $\kappa>0$ such that 
	\begin{equation*}
	\dist(x,L\cap \stdCone) \leq \kappa (\dist(x,L)+\dist(x,\stdCone))\quad \forall x\in \R^n.
	\end{equation*}
\end{proposition}
\begin{proof} It follows from \cite[Corollary~3]{BBL99} that $L$ and $\stdCone$ are boundedly linearly regular, i.e. for every bounded set $S$ there exists a $\kappa_S>0$ that satisfies the error bound
	\begin{equation}\label{eq:blr}
	\dist(x,L \cap \stdCone)\leq \kappa_S \max \{ \dist(x,L), \dist(x,\stdCone)\}\quad \forall x\in S. 
	\end{equation}
	Moreover, because $L$ and $\stdCone$ are  closed convex cones, by homogeneity one can choose $\kappa_S: = \bar \kappa$ in \eqref{eq:blr} independent of $S$ (see \cite[Theorem~10]{BBL99}). 
\end{proof}

\subsection{Hyperbolic polynomials and hyperbolicity cones}
\label{sec:prelim-hyp}
In this section, we recall some basic facts about hyperbolic polynomials, hyperbolicity cones, and 
their derivative relaxations. Recall that if $p$ is hyperbolic with respect to $e$ then we 
write $\Lambda_+(p,e)$ for the associated hyperbolicity cone. It turns out that the dependence of 
the hyperbolicity cone on the choice of $e$ is quite weak. In fact, if $z\in \textup{int}(\Lambda_+(p,e))$
then $p$ is hyperbolic with respect to $z$, and $\Lambda_+(p,e) = \Lambda_+(p,z)$, a result of 
G\r{a}rding~\cite{gaarding1959inequality}. When $p$ and $e$ are clear from the context, 
we will omit them and write $\Lambda_+$. 

\paragraph{Polynomial inequality description of hyperbolicity cones}
The definition of the closed hyperbolicity cone $\Lambda_+(p,e)$ given in~\eqref{eq:hypcone}
is expressed in terms of the nonnegativity of the zeros of $p(te-x)$. It is also useful 
to work with an alternative description of $\Lambda_+(p,e)$ in terms of polynomial inequalities. 
Let $D_e^kp(x)= \left.\frac{d^k}{dt^k}p(x+te)\right|_{t=0}$ denote the $k$th directional derivative
of $p$ in the direction $e$. 
Expanding $p(x+te)$ in powers of $t$ gives
\[ p(x+te) = p(x) + tD_ep(x) + \frac{t^2}{2{!}}D_e^2p(x) + \cdots + \frac{t^d}{d{!}}D_e^dp(x).\]
Since a real-rooted polynomial has all roots non-positive if and only if it has all coefficients
non-negative, we have 
\begin{equation}
	\label{eq:hyp-ineq}
	\Lambda_+(p,e) = \{x\in \RR^n\;:\; p(x) \geq 0,\; D_ep(x) \geq 0,\; \ldots,\; D_e^{d-1}p(x) \geq 0\}.
\end{equation}
(This description can also be obtained, for instance, by combining Proposition~18 
and Theorem~20 of~\cite{Re06}.) The interior of $\Lambda_+(p,e)$ is then 
\begin{equation}
	\label{eq:hyp-int}
	\textup{int}(\Lambda_+) = \{x\in \RR^n\;:\; p(x) > 0,\; D_ep(x) > 0,\; \ldots,\; D_e^{d-1}p(x) > 0\}.
\end{equation}

\paragraph{Derivative relaxations}
If $p$ is hyperbolic with respect to $e$, 
then so are the directional derivatives $D_e^mp(x)$ (by Rolle's theorem, see also  \cite[Proposition~18]{Re06}). 
Denote the associated hyperbolicity 
cone by $\Lambda_+(D_e^mp,e)$. If $p$ and $e$ are clear from the context we write $\Lambda_+^{(m)}\coloneqq\Lambda_+(D_e^mp,e)$. If $m\leq d-1$ then it follows from~\eqref{eq:hyp-ineq} that 
\begin{equation}
	\label{eq:deriv-ineq}
\Lambda_+^{(m)} = \{x\in \RR^n\;:\; D_e^mp(x) \geq 0,\; D_e^{m+1}p(x) \geq 0, \;\ldots,\; D_e^{d-1}p(x) \geq 0\}.
\end{equation}
For convenience we define $\Lambda_+^{(d)} = \RR^n$. From this description it is clear that 
\begin{equation}
\label{eq:relax}
\Lambda_+ \subseteq \Lambda_+^{(1)} \subseteq \cdots \subseteq 
\Lambda_+^{(d-1)} \subseteq \Lambda_+^{(d)} = \RR^n.
\end{equation}
These are known as \emph{Renegar derivatives} or, in light of~\eqref{eq:relax}, 
\emph{derivative relaxations} of $\Lambda_+$.

The following (standard) technical fact about the coefficients of univariate polynomials 
with nonnegative coefficients will be useful later in our arguments.
\begin{lemma}\label{lem:p_coef}
	Let $p(t) = a_dt^d + a_{d-1}t^{d-1} + \cdots+ a_1t + a_0$ 
	be a univariate polynomial with real zeros, nonnegative coefficients, and $a_d>0$. 
	If $a_k=0$ for some $k<d$, then $a_\ell = 0$ for all $\ell  < k$.
\end{lemma}
\begin{proof}
	Because all coefficients are nonnegative and $a _d > 0$, we have $p(\lambda) > 0$ for every positive $\lambda$.
	Therefore, all roots of $p$ are nonpositive. Let $\lambda_1,\ldots,\lambda_d\leq 0$ be the roots of $p$.
	
	We recall that $|a_k/a_d| = |e_{d-k}(\lambda_1,\ldots,\lambda_d)|$ is 
	the absolute value of the elementary symmetric polynomial of degree $d-k$ in the roots of $p$. 
	Since all the roots of $p$ are non-positive, 
	$|e_{d-k}(\lambda_1,\ldots,\lambda_d)|= e_{d-k}(|\lambda_1|,\ldots,|\lambda_d|)$.
	Therefore, if $a_k = 0$, it follows that 
	every product of $d-k$ roots of $p$ must vanish.
	This shows that $a_{\ell} = 0$ for $\ell < k$, because $|a_{\ell}/a_d| = e_{d-\ell}(|\lambda_1|,\ldots,|\lambda_d|)$ 
	is the absolute value of a sum of products of more than $d-k$ roots of $p$.	
\end{proof}

\section{Faces of hyperbolicity cones}
\label{sec:faces-amenability}
In this section, we show that 
the faces of hyperbolicity cones can be expressed as the intersection of a linear 
space with a suitable derivative relaxation of the hyperbolicity cone, 
in such a way that the linear space meets the interior of the derivative 
relaxation. This new representation of the faces of hyperbolicity cones is 
the crucial ingredient in our proof that every hyperbolicity cone is amenable, which 
is the focus of Section~\ref{sec:pf-hyp_am}.

\subsection{Facial structure of hyperbolicity cones}
\label{sec:hyp-faces}

Given $x\in \RR^n$, let $\textup{mult}_{p}(x)$ denote the 
multiplicity of $0$ as a zero of $t\mapsto p(te-x)$. 
It turns out that this quantity is independent of the 
choice of $e$~\cite[Proposition 22]{Re06}, so we suppress $e$ from the notation.

The following result tells us that a point of multiplicity $m$ in a hyperbolicity cone 
is in the interior of the $m$th derivative relaxation. 
\begin{lemma}
	\label{lem:iz}
	Let $p$ be hyperbolic with respect to $e$. If $\textup{mult}_{p}(x)=m$ 
	and $x\in \Lambda_+(p,e)$ then 
	\begin{itemize}
		\item $D_e^{k}p(x) = 0$ for $k=0,1,\ldots,m-1$ and
		\item $D_e^kp(x) > 0$ for $k=m,m+1,\ldots, d$,
	\end{itemize}
	and so $x\in \textup{int}(\Lambda_+^{(m)})$. 
\end{lemma}
\begin{proof}
	Since $\textup{mult}_{p}(x) = m$, we know that $t\mapsto p(te-x)$ 
	vanishes to exactly order $m$ at $t=0$. Since $p(te-x)= (-1)^dp(x+(-t)e)$, 
	it also holds that $t\mapsto p(x+te)$ vanishes to exactly order $m$ at $t=0$.
	Hence $D_e^\ell p(x) = 0$ for $\ell=0,1,\ldots,m-1$, and 
	and $D_e^mp(x)\neq 0$ (otherwise $\textup{mult}_{p}(x) > m$). As such,
	\[ p(x+te) = \frac{t^m}{m{!}}D_e^mp(x) + \frac{t^{m+1}}{(m+1){!}}D_e^{m+1}p(x) + \cdots + \frac{t^d}{d{!}}D_e^dp(x).\]
	Since $x\in \Lambda_+$, we know that $D_e^kp(x) \geq 0$ for all $k$ and 
	so, since $D_e^mp(x)\neq 0$, it must be the case that $D_e^mp(x) > 0$.
	By Lemma~\ref{lem:p_coef}, it follows that 
	$D_e^kp(x) > 0$ for $k=m,m+1,\ldots, d-1$.
	Moreover $D_e^dp(x) = d{!}p(e) > 0$, since $p$ is hyperbolic with respect to $e$. 
	Since $D_e^kp(x)>0$ for $k=m,m+1,\ldots,d-1$, it follows from~\eqref{eq:deriv-ineq} and~\eqref{eq:hyp-int}
	that $x\in \textup{int}(\Lambda_+^{(m)})$.
\end{proof}
We also require the following key fact, which is part of~\cite[Theorem 26]{Re06}, 
about multiplicity and faces of hyperbolicity cones.
\begin{lemma}
	\label{lem:mult-constant}
	Suppose that $p$ is hyperbolic with respect to $e$. 
	If $\stdFace$ is a face of $\Lambda_+(p,e)$ and $x,z\in \reInt(\stdFace)$ then 
	$\textup{mult}_p(x) = \textup{mult}_p(z)$. 
\end{lemma}
We now establish the main result of this subsection, which gives a new description of the faces of 
hyperbolicity cones. 
\begin{proposition}[Faces and derivative relaxations]
	\label{prop:face-rep}
	Suppose that $p$ is hyperbolic with respect to $e$ and $\stdFace$ is 
	a face of $\Lambda_+(p,e)$. Let $z\in \reInt(\stdFace)$ and let $m = \textup{mult}_p(z)$. 
	Then $z\in \textup{int}(\Lambda_+^{(m)})$ and  
	\[ \stdFace = \lspan(\stdFace) \cap \Lambda_+^{(m)}.\]
\end{proposition}
\begin{proof}
	Since $\textup{mult}_p(z) = m$ and $z\in \reInt(\stdFace)\subseteq \Lambda_+$, 
	it follows from Lemma~\ref{lem:iz} that $p(z) = D_ep(z) = \cdots = D_e^{m-1}p(z) = 0$
	and that $z\in \textup{int}(\Lambda_+^{(m)})$. 
	
	Since $z\in \reInt(\stdFace)$ has multiplicity $m$, it follows from Lemma~\ref{lem:mult-constant} that 
	every point in the relative interior of $\stdFace$ has multiplicity $m$. As such
	$p(x) = D_ep(x) = \cdots = D^{m-1}_ep(x) = 0$ for all $x\in \reInt(\stdFace)$. 
	Since $\reInt(\stdFace)$ is full-dimensional in $\lspan(\stdFace)$, it follows that
	$p(x) = D_ep(x) = \cdots = D^{m-1}_ep(x) = 0$ for all $x\in \lspan(\stdFace)$. 
	
	Since $\stdFace$ is a face of $\Lambda_+$, we have  that 
	$\stdFace = \lspan(\stdFace) \cap \Lambda_+$. 
	It follows from~\eqref{eq:relax} that 
	$\stdFace = \lspan(\stdFace) \cap \Lambda_+ \subseteq \lspan(\stdFace)\cap \Lambda_+^{(m)}$. 
	For the reverse inclusion, suppose that $x\in \lspan(\stdFace)\cap \Lambda_+^{(m)}$. Then 
	$p(x) = D_e^1p(x) = \cdots = D_e^{m-1}p(x) = 0$ (since $x\in \lspan(\stdFace)$) and 
	$D_e^kp(x)\geq 0$ for $k=m,m+1,\dots,d-1$ (since $x\in \Lambda_+^{(m)}$). 
	Hence  $x\in \Lambda_+$ (from~\eqref{eq:hyp-ineq})
	and so $\lspan(\stdFace)\cap \Lambda_+^{(m)} \subseteq \lspan(\stdFace)\cap \Lambda_+ = \stdFace$.
\end{proof}
This result can be rephrased as saying that 
faces of hyperbolicity cones are, themselves, hyperbolicity cones. To see this, first we clarify the notation
we will use. If $L\subseteq \Re^n$ is a subspace and $p$ is hyperbolic with respect to $e\in L$, 
then the restriction of $p$ to $L$, denoted $\left.p\right|_{L}$,
is also hyperbolic with respect to $e$ when thought of as a polynomial on $L$. 
The corresponding hyperbolicity cone, $\Lambda_+(\left.p\right|_L,e)$, is linearly isomorphic to 
$\Lambda_+(p,e) \cap L$.
\begin{corollary}\label{cor:faces}
	Suppose that $p$ is hyperbolic with respect to $e$ and $\stdFace$ is a face of $\Lambda_+(p,e)$. If $z\in \reInt(\stdFace)$ 
	and $m=\textup{mult}_p(z)$ then $q = \left.D_e^{m}p\right|_{\lspan(\stdFace)}$ is hyperbolic with respect to $z$ and 
	$\stdFace = \Lambda_+(q,z)$.
\end{corollary}

We can use Corollary~\ref{cor:faces} to understand the intersection of two hyperbolicity cones. Let 
$\Lambda_+(p_1,e_1)$ and $\Lambda_+(p_2,e_2)$ be two 
hyperbolicity cones. If their relative interiors intersect, then $\Lambda_+(p_1,e_1)\cap  \Lambda_+(p_2,e_2)$ is also a hyperbolicity cone and the corresponding hyperbolic polynomial is $p_1p_2$ restricted to the span of $\Lambda_+(p_1,e_1)\cap  \Lambda_+(p_2,e_2)$. However, when $(\reInt\Lambda_+(p_1,e_1))\cap  (\reInt\Lambda_+(p_2,e_2)) = \emptyset$, the situation is less obvious.
As a consequence of Corollary~\ref{cor:faces}, we can show that the intersection of two hyperbolicity cones is a hyperbolicity cone. 
\begin{corollary}\label{cor:intersection}
Suppose that $p_1$ and $p_2$ are hyperbolic polynomials with respect to $e_1$ and $e_2$, respectively. Then, $\Lambda_+(p_1,e_1)\cap  \Lambda_+(p_2,e_2)$ is a hyperbolicity cone.
\end{corollary}
\begin{proof}
Let $\stdCone = \Lambda_+(p_1,e_1)\cap  \Lambda_+(p_2,e_2)$.
It is a well-known fact from convex analysis that if $\stdFace _i$ is the smallest face of 
$\Lambda_+(p_i,e_i)$ containing $\stdCone$, 
then we have 
\begin{equation*}
	\reInt \stdCone \subseteq \reInt\stdFace_i \qquad \textup{for $i =1,2$,}
\end{equation*}
	and $\stdCone = \stdFace_1 \cap \stdFace_2$. (See, for example, \cite[Proposition~2.2]{LRS20}.)
In particular, there exists $z \in \reInt \stdCone$ such that  $z \in (\reInt\stdFace_1)\cap(\reInt \stdFace_2)$.

For $i = 1,2$, let $m_i \coloneqq \textup{mult}_{p_i}(z)$ and $q_i \coloneqq \left.D_{e_i}^{m_i}p_i\right|_{\lspan(\stdFace_i)}$. By Corollary~\ref{cor:faces}, 
$q_i$ is hyperbolic with respect to $z$ and
$\stdFace_i = \Lambda_+(q_i,z)$.
If $L = \lspan \stdCone$, then 
$\left.(q_1q_2)\right|_{L}$ is hyperbolic with respect to $z$ and the corresponding hyperbolicity cone is $\stdCone = \stdFace_1\cap \stdFace_2$.
\end{proof}

\subsection{Explicit description of the span of a face}
\label{sec:sp-f}
	In Section~\ref{sec:hyp-faces} we gave a description of a face of a hyperbolicity
	cone as the intersection of the span of the face with an appropriate
	derivative relaxation. While Proposition~\ref{prop:face-rep}
	will be sufficient to establish that hyperbolicity cones are amenable, 
	the expression therein is implicit in the sense that in describing a face, 
	it refers to the span of that 
	face. In this section we show how to give an explicit description of the span 
	of a face of a hyperbolicity cone by using results from~\cite{Re06}. The description
	only uses the hyperbolic polynomial, a direction of hyperbolicity, and a point 
	in the relative interior of the face.

	Suppose that $q$ is hyperbolic with respect to $e$ and $z\in \Lambda_+(q,e)$ 
	has $\textup{mult}_q(z) = 1$. This means that $z$ is a smooth point of the boundary of the 
	hyperbolicity cone. Following~\cite{Re06} let
	\begin{equation}\label{eq:tan} 
	T_z\Lambda_+(q,e) = \{x\;:\; D_xq(z) = 0\}
	\end{equation}
	denote the tangent space to the hyperbolicity cone at the point $z$. 
	
	Our explicit description of the span of a face is closely related to
	the following result of Renegar.
	\begin{theorem}[{Renegar~\cite[Proposition 9 and Theorem 10]{Re06}}]
		\label{thm:renegar-hessian}
		Suppose that $z\in \Lambda_+(q,e)$ and $\textup{mult}_q(z)=1$ and $x\in T_z\Lambda_+(q,e)$. Then $D^2_xq(z) \leq 0$. Moreover, exactly one of the following holds:
		\begin{itemize}
			\item $q(z+tx) = 0$ for all $t$ and there exists $\epsilon>0$ such that 
				$z+tx \in \Lambda_+(q,e)$ and $\textup{mult}_q(z+tx)=1$ whenever $t\in (-\epsilon,\epsilon)$;
			\item $D_x^2q(z) < 0$.
		\end{itemize}
	\end{theorem}
	At this point, it is convenient to introduce one more piece of notation. 
	Suppose that $z\in \Lambda_+(q,e)$ has multiplicity one with respect to $q$
	and that the associated tangent space is $T_z\Lambda_+(q,e)$. We use 
	the notation $\left.H_q(z)\right|_{T_z}:T_z\Lambda_+(q,e)\rightarrow T_z\Lambda_+(q,e)$ 
	to denote the self-adjoint linear map associated with the restriction of 
	the quadratic form $x\mapsto D^2_xq(z)$ to the subspace $T_z\Lambda_+(q,e)$. (See Section~\ref{sec:ex} for concrete examples.) Equivalently, we can also think of $\left.H_q(z)\right|_{T_z}$ as 
	the Hessian of the polynomial $\tilde{q} = \left. q\right|_{T_z\Lambda_+(q,e)}$ obtained 
	by restricting $q$ to the subspace $T_z\Lambda_+(q,e)$. 

	We now give an explicit description of the 
	span of a face of a hyperbolicity cone. This can be combined 
	with Proposition~\ref{prop:face-rep} to give an explicit description of a face of 
	a hyperbolicity cone, given a point in the relative interior of the face.
\begin{proposition}
	\label{prop:spf}
	Suppose that $p$ is hyperbolic with respect to $e$ and $\stdFace$ is a face
	of $\Lambda_+(p,e)$. Let $z\in \reInt(\stdFace)$, let $m= \textup{mult}_p(z)$,
	and let $q(x) = D^{m-1}_ep(x)$. 
	Then 
	\begin{equation}\label{eq:prop:spf}
	 \lspan(\stdFace) = \{x\in T_z\Lambda_+(q,e) \;:\; \left.H_q(z)\right|_{T_z} x = 0\}.
	\end{equation}
\end{proposition}
\begin{proof}
	If $\stdFace$ is a face of $\Lambda_+(p,e)$ of multiplicity $m$ then 
	$\stdFace$ is also a face of $\Lambda_+(q,e)$ of multiplicity $1$~\cite[Proposition 25]{Re06}. 
	As such, it is enough to establish the result when $m=1$ and $p=q$. 

	For the ``$\supseteq$'' inclusion, assume that $x\in T_z\Lambda_+(q,e)$ and 
	that $\left.H_q(z)\right|_{T_z}x = 0$. It follows that $D_x^2q(z) = 0$. 
	Since $z\in \Lambda_+(q,e)$ and $\textup{mult}_q(z)=1$, it follows from 
	Theorem~\ref{thm:renegar-hessian} that $q(z+tx) = 0$ for all $t$ and 
	that there exists $\epsilon>0$ such that 
	$z+tx\in \Lambda_+(q,e)$ for all $t\in (-\epsilon,\epsilon)$. 
	Let $\hat{\epsilon} = \epsilon/2$.
	Since $\stdFace$ 
	is a face of $\Lambda_+(q,e)$, $z\pm\hat{\epsilon}x \in \Lambda_+(q,e)$  and $2z = (z-\hat{\epsilon} x)+ (z+\hat{\epsilon} x)$, it follows that 
	$z-\hat{\epsilon} x,z+\hat{\epsilon} x\in \mathcal{F}$.
	Taking a linear combination of these points we see that 
	$[(z+\hat{\epsilon} x)-(z-\hat{\epsilon} x)]/(2\hat{\epsilon}) = x\in \lspan(\stdFace)$.

	We now consider the reverse inclusion. 
	First, we will check that 
	\begin{equation}\label{eq:renegar_face}
	\stdFace  \subseteq T_z\Lambda_+(q,e).
	\end{equation}
	Let $y \in \stdFace$. Since $z+ty\in \stdFace$ for all $t\geq 0$ and $\stdFace$ is a face of multiplicity
	one, it follows that $z+ty$ has multiplicity at least one for all $t \geq 0$.
	This implies that $q(z+ty) = 0$ for all $t \geq 0$. Since $q$ is a polynomial, we have that, in fact, 
	$q(z+ty) = 0$ for all $t \in \Re$. Therefore 
	$D_yq(z) = 0$, which shows that $y \in T_z\Lambda_+(q,e)$, by the definition of 
	$T_z\Lambda_+$ in~\eqref{eq:tan}. Furthermore, we can also conclude that $D_y^2q(z) = 0$.

		It remains to show that $y\in \stdFace$ implies that $\left.H_q(z)\right|_{T_z}y = 0$.
		Now, the quadratic form $x\mapsto D_x^2q(z)$ is negative semidefinite 
		restricted to $T_z\Lambda_+(q,e)$ (from Theorem~\ref{thm:renegar-hessian}) and 
		$y \in T_z\Lambda_+(q,e)$ by \eqref{eq:renegar_face}. Therefore, $D_y^2q(z) = 0$ implies that 
		$y$ is in the nullspace of the associated linear map, which is  $\left.H_q(z)\right|_{T_z}$.
		This establishes that $\stdFace$ is contained in the right-hand side of \eqref{eq:prop:spf}.
		 Since the right-hand side of \eqref{eq:prop:spf} is a subspace, it follows 
		 that it also contains $\lspan \stdFace$.

\end{proof}

\subsection{Illustrative examples}\label{sec:ex}
\begin{figure}
	\begin{center}
	\begin{tikzpicture}
		\node at (0,0) {\includegraphics[scale=0.2]{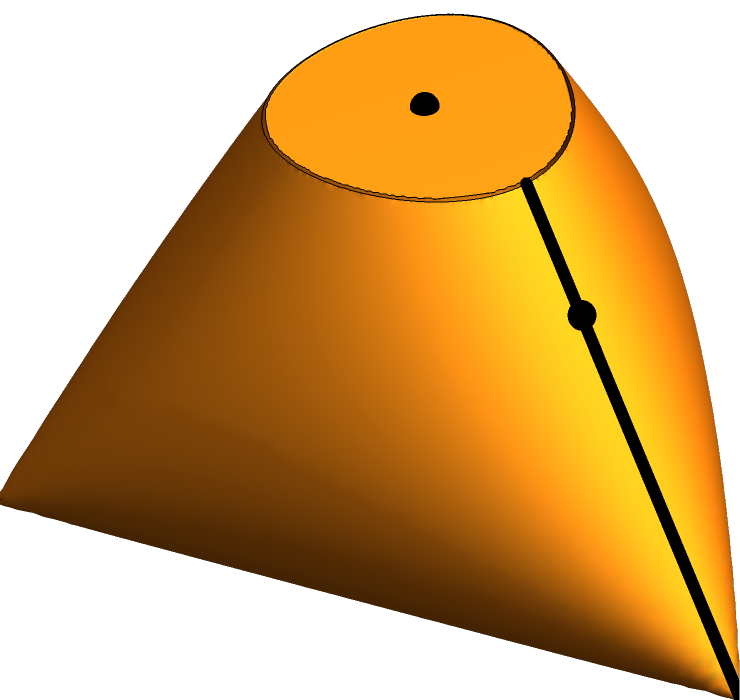}};
		\node at (5,0) {\includegraphics[scale=0.5]{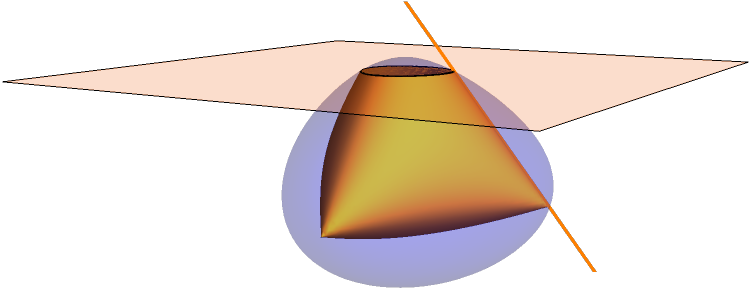}};
		\node at (0.2,1) {$\tilde{z}$};
		\node at (0.5,0) {$\tilde{w}$};
	\end{tikzpicture}
	\end{center}
	\caption{\label{fig:ill} On the left is an affine slice of $\Lambda_+(q,e)$, showing the points
	$\tilde{z}$ and $\tilde{w}$ in the relative interiors of the faces of interest. On the right are (rotations of)
	the same affine slices of $\Lambda_+(q,e)$, the relaxation $\Lambda_+(D_eq,e)$, and the spans of 
	the two faces of interest.}
\end{figure}
In this section we apply the results of Propositions~\ref{prop:face-rep} and~\ref{prop:spf} to 
give explicit descriptions of two faces of the hyperbolicity cone
associated with the polynomial 
\[ q(x_1,x_2,x_3,x_4) = (x_1x_2x_3+x_1x_2x_4+x_1x_3x_4+x_2x_3x_4)(x_1+x_2+x_3-x_4/2). \]
The polynomial $q$ is hyperbolic with respect to $e=(1/4,1/4,1/4,1/4)$. For the
purposes of visualization, we consider the three dimensional affine section
$x_1+x_2+x_3+x_4=1$ of the hyperbolicity cone $\Lambda_+(q,e)$, under the affine
isomorphism that sends the standard basis vectors in $\RR^4$ to the vertices 
of a regular tetrahedron. The resulting convex body is shown in Figure~\ref{fig:ill}. 
It is the intersection of the elliptope with an affine halfspace.

As a first example, consider the face with $\tilde{z}$ in its relative interior, 
shown in Figure~\ref{fig:ill}.
This corresponds to a three-dimensional face of the cone $\Lambda_+(q,e)$ 
with the point $z = (1/9,1/9,\linebreak1/9,2/3)$ in its relative interior. This point has
multiplicity one with respect to $q$, because only the linear
factor of $q$ vanishes at $z$. 
Proposition~\ref{prop:face-rep} tells us that
the face is the intersection of its span with the derivative relaxation
$\Lambda_+(D_eq,e)$. This derivative relaxation, together with the span 
of the face, are shown in Figure~\ref{fig:ill}. 

It is clear from Figure~\ref{fig:ill} that the span of the face
is just the hyperplane corresponding to the linear factor in $q$, i.e., 
$\{(x_1,x_2,x_3,x_4)\;:\; x_1+x_2+x_3-x_4/2=0\}$. We can also 
obtain this from Proposition~\ref{prop:spf}.  Since 
\[ \nabla q(z) = \frac{19}{1458}\begin{bmatrix}2\\2\\2\\-1\end{bmatrix}\;\;\textup{and}\;\; 
	D_x^2q(z) = \frac{26}{81}\left(x_1+x_2+x_3-\frac{x_4}{2}\right)\left(x_1+x_2+x_3+\frac{3x_4}{13}\right),\]
	it follows that $T_{z}\Lambda_+(q,e)  = \{(x_1,x_2,x_3,x_4)\;:\; x_1+x_2+x_3-x_4/2 = 0\}$. Consequently the restriction of $x\mapsto D_x^2q(z)$ to $T_z\Lambda_+(q,e)$ 
	is identically zero, and so $\left.H_q(z)\right|_{T_z} = 0$.
	Hence Proposition~\ref{prop:spf} tells us that 
	the span of the face is just $T_z\Lambda_+(q,e)$ in this case.

	As a second example, consider the face with 
	$\tilde{w}$ in its relative interior, shown in Figure~\ref{fig:ill}.
	This corresponds to a two-dimensional
	face of $\Lambda_+(q,e)$ with the point $w = (1/2,0,0,1/2)$ in its
	relative interior. This point has multiplicity one with respect to $q$
	since the cubic factor of $q$ vanishes to order one at $w$ and the linear
	factor does not vanish at $w$. 
	Again Proposition~\ref{prop:face-rep} tells us that the face is the intersection of
	its span with the derivative relaxation $\Lambda_+(D_eq,e)$, 
	as shown in Figure~\ref{fig:ill}. 

	We now show how to compute the span of this face using 
	Proposition~\ref{prop:spf}. Since 
	\[ \nabla q(w) = \frac{1}{16}\begin{bmatrix} 0\\1\\1\\0\end{bmatrix}\quad\textup{and}\quad D_x^2q(w) = \frac{3x_1}{4}\left(x_2+x_3\right) + \frac{1}{2}\left(x_2^2+3x_2x_3+x_3^2\right)\]
		it folllows that 
		$T_w\Lambda_+(q,e)= \{(x_1,x_2,x_3,x_4)\;:\; x_2+x_3=0\}$.
		If we express elements of $T_w\Lambda_+(q,e)$ in the form
		$(y_1,y_2,-y_2,y_3)$ for $y\in \RR^{3}$ then we have that 
		\[D_{(y_1,y_2,-y_2,y_3)}^2q(w) = -y_2^2/2\quad\textup{and so}\quad
		\left.H_w(q)\right|_{T_w}  
			= \begin{bmatrix} 0 & 0 & 0\\0 & -1/2 & 0\\0 & 0 & 0\end{bmatrix}.\]
				Clearly, then, $\left.H_w(q)\right|_{T_w}y = 0$
				if and only if $y_2=0$. 
				From Proposition~\ref{prop:spf} we get that 
				\[ \textup{span}(\mathcal{F}) = 
				\{(y_1,y_2,-y_2,y_3)\;:\; y_2=0,\; y_1\in \RR,\;y_3\in \RR\} = \{(y_1,0,0,y_3)\;:\; y_1\in \RR,\;y_3\in \RR\}\]
				which is just the span of the first and last 
				coordinate directions.

\section{Amenability of hyperbolicity cones}
\label{sec:pf-hyp_am}

In this brief section we establish that every hyperbolicity cone is amenable
(Theorem~\ref{thm:hyp_am}).  The argument combines the standard error bound
from Proposition~\ref{prop:error} and our new representation of faces of
hyperbolicity cones from Proposition~\ref{prop:face-rep}.
\begin{proof}[{Proof of Theorem~\ref{thm:hyp_am}}]
	Let $p$ be hyperbolic with respect to $e\in \RR^n$, and let $\stdFace \face \Lambda_+(p,e)$ be a face of 
	the associated hyperbolicity cone. By Proposition~\ref{prop:face-rep}, 
	there exists some derivative relaxation $\Lambda_+^{(m)}$ satisfying the following two properties:
	\[
		\stdFace \cap (\interior \Lambda_+^{(m)} ) \neq \emptyset \quad\textup{and}\quad \stdFace = \lspan(\stdFace) \cap \Lambda_+^{(m)}.
	\]
	In particular, $\lspan (\stdFace) \cap (\interior \Lambda_+^{(m)} ) \neq \emptyset $, and so by 
	Proposition~\ref{prop:error}, there exists $\kappa > 0$ such that for all $x \in \Re^n$ we have
	\begin{align*}
	\dist(x, \stdFace) &\leq \kappa \max(\dist(x,\lspan \stdFace),\dist(x,  \Lambda_+^{(m)} ))\\
	& \leq \kappa \max (\dist(x,\lspan \stdFace),\dist(x,  \Lambda_+(p,e)),
	\end{align*}
	where the last inequality follows from the 
	fact that $\Lambda_+(p,e) \subseteq \Lambda_+^{(m)}$.
	This implies that\[
		\dist(x, \stdFace) \leq \kappa \dist(x,  \Lambda_+(p,e)) \quad \textup{for all $x \in \lspan \stdFace$.}
	\]
	Therefore $\Lambda_+(p,e)$ is amenable.
\end{proof}

\section{Discussion}
\label{sec:discussion}
It could be fruitful to examine in more depth the geometry and facial structure of hyperbolicity cones 
and check how it compares with spectrahedral cones. One recent result in this style is that the 
tangent cone to a hyperbolicity cone at a point is, again, a hyperbolicity 
cone~\cite[Theorem 5.9]{saunderson2020terracini}, which parallels the analogous result for spectrahedral 
cones. For a direction where many questions remain open, it might be interesting to
see how spectrahedral cones and hyperbolicity cones fare with regards to
\emph{projectional exposedness}. 

We say that a closed convex cone $\stdCone$ is {projectionally exposed} if for every face $\stdFace \face \stdCone$ there exists a  projection $P$ (not necessarily orthogonal) such that $P(\stdCone) = \stdFace$, see \cite{BW81,ST90}.
This notion appears in connection to facial reduction algorithms and approaches for regularizing convex conic optimization problems. It turns out that projectional exposedness implies amenability \cite[Proposition~9]{L19}
and in dimension four or less, amenability implies projectional exposedness \cite[Corollary~6.4]{LRS20}.

So far, the largest classes of hyperbolicity cones which are known to be projectionally exposed are the symmetric cones (\cite[Proposition~33]{L19}) and the polyhedral cones (\cite[Corollary 3.4]{ST90}). However, \cite[Corollary~6.4]{LRS20} and Theorem~\ref{thm:hyp_am} imply that all hyperbolicity cones in dimension four or less are projectionally exposed, which is a curious fact that, at this moment, does not seem to have an obvious algebraic explanation.

\section*{Acknowledgements}
Bruno F. Louren\c{c}o is supported partly by the JSPS Grant-in-Aid for Early-Career Scientists 19K20217 and the Grant-in-Aid for Scientific Research (B)18H03206.

Vera Roshchina is grateful to the Australian Research Council for continuing support. Specifically, her ARC DECRA grant DE150100240 helped sustain the initial collaboration on this project.

James Saunderson is the recipient of an Australian Research Council Discovery Early Career Researcher Award 
(project number DE210101056) funded by the Australian Government.

\bibliographystyle{abbrvurl}
\bibliography{refs}
\end{document}